\documentclass[11pt]{article}
\usepackage{amsmath}
\usepackage{amssymb,amsbsy,amsmath,amsfonts,amssymb,amscd,amsthm}
\usepackage{stmaryrd}
\usepackage{mathrsfs}
\usepackage[active]{srcltx}
\usepackage{bbm}
\usepackage{amssymb}
\usepackage{graphicx}
\usepackage{color}

\usepackage[dvipsnames]{xcolor}
\usepackage[colorlinks,linkcolor={blue},citecolor={blue},urlcolor={purple},]{hyperref}

\theoremstyle{plain}
\newtheorem{theorem}{Theorem}[section]
\theoremstyle{remark}
\newtheorem{remark}[theorem]{Remark}

\theoremstyle{plain}

\newtheorem{lemma}[theorem]{Lemma}
\newtheorem{proposition}[theorem]{Proposition}

\numberwithin{equation}{section}

\def\a{\alpha}

\def\g{\gamma}
\def\m{\mu}
\def\l{\lambda}

\def\a{\alpha}

\def\g{\gamma}

\def\l{\lambda}

\def\cI{{\cal I}}

\def\cR{{\cal R}}

\def\Re {\cR e }
\def\Im {\cI m }

 \numberwithin{equation}{section}

\usepackage[markup=underlined]{changes}
\definechangesauthor[color=blue]{ck}

\usepackage{color}
\usepackage{xcolor}

\definecolor{darkgreen}{RGB}{10,160,20}

\usepackage{caption}
\usepackage{subcaption}

\usepackage{enumitem}

\begin{document}
\title{Stability of the Abstract  Thermoelastic System with Singularity
		\thanks{The first author is supported by the China Scholarship Council (grant No. 202206030056). The second author is supported by the National Natural Science Foundation of China (grant No. 62073236). The fourth author is supported by the National Natural Science Foundation of China (grants No. 12271035, 12131008) and Beijing Municipal Natural Science Foundation (grant No. 1232018).}
		
	}
	\date{}

		\author{Chenxi Deng\thanks{School of Mathematics and Statistics,
				Beijing Institute of Technology, Beijing 100081, China (email: chenxideng@bit.edu.cn)}
		,\quad Zhong-Jie Han\thanks{{School of Mathematics, BIIT Lab, Tianjin University,
				Tianjin 300354,   China {(email: zjhan@tju.edu.cn)}}
		},\quad Zhaobin Kuang\thanks{{Computer Science Department, Stanford University, Stanford 94305, U.S.A.{(email: zhaobin.kuang@gmail.com)}}
	}, \quad
		Qiong Zhang\thanks{School of Mathematics and Statistics, Beijing Key Laboratory on MCAACI,
		Beijing Institute of Technology, Beijing 100081, China (email: zhangqiong@bit.edu.cn)}
  \thanks{Corresponding author.}
}
	\maketitle

	\begin{abstract}
In this paper, we analyze an abstract thermoelastic system,
where the heat conduction follows the Cattaneo law. Zero becomes a spectrum point of the system operator when the coupling and thermal damping parameters of system satisfy specific conditions.
We obtain the decay rates of solutions to the system with or without the
inertial term. Furthermore, the decay
rate of the system without inertial terms is shown to be optimal.

	\end{abstract}
	
 {\bf Key words:}\hspace{1mm} thermoelastic system, stability, inertial term, singularity.
 \vspace{3mm}

  {\bf AMS subject classifications:}\hspace{2mm} 35Q74, 74F05.


\section{Introduction}

 In this paper, we study an abstract thermoelastic system in which the heat conduction follows the Cattaneo law. The Cattaneo law (\cite{cattaneo,vernotte}) describes finite heat propagation speed in a medium, which resolves the paradox of infinite speed of heat transfer in Fourier law and characterizes the wave-like motion of heat, known as the second sound in physics. The abstract thermoelastic system reads as follows:
 \begin{equation}\label{101}
	\left\{
	\begin{array}{l}
	u_{tt} +mA^\gamma  u_{tt} +\sigma A u -A^\alpha\theta =0,  \\ \noalign{\medskip}  \displaystyle
	\theta_t -A^{\frac{\beta}{2}}q +A^\alpha u_t =0,  \\ \noalign{\medskip}  \displaystyle
	\tau q_t+q+A^{\frac{\beta}{2}}\theta=0,  \\ \noalign{\medskip}  \displaystyle
	u(0)= u_0, \;\; u_t(0) = u_1,\;\; \theta(0) =\theta_0,\;\; q(0) = q_0,
	\end{array} \right.
	\end{equation}
where $A$ is a self-adjoint, positive definite operator with compact resolvent on a Hilbert space $H$ {equipped} with inner product $(\cdot,\,\cdot)$ and the induced norm $\|\cdot\|$.
Parameters $  \alpha, \;\beta,\;\gamma  $   represent coupling, thermal damping, and inertial characteristics, respectively.
Here we assume $(\alpha, \;\beta,\;\gamma) \in Q$, where $$ Q :=\Big\{(\alpha, \;\beta,\;\gamma) \in [0,1]\times [0,1]\times(0,1]\;\Big|\:  \alpha>{\beta+1\over2} \Big\}.$$
 {We further assume $m$ is non-negative. In fact, $m>0$ and $m=0$ indicate the system includes and excludes inertial term, respectively.} Note that we omit the case $\gamma=0$ since  it can be encompassed within the
case $m = 0$. {Let} $\sigma>0$ { denote} wave speed, $\tau\ge0$ {denote} the  the relaxation parameter of heat conduction. {Particularly}, the heat conduction follows Cattaneo law when $\tau >0$, and Fourier law when $\tau=0.$

There are {several studies investigating} the long time behavior or regularity of thermoelastic system (\ref{101}) under the Fourier heat conduction mechanism, i.e., {when $\tau=0$ in (\ref{101})}. {We refer the readers to \cite{Ammar,  haoliu1,haoliu2,kim, LL,rivera1} for the case $m=0$ and \cite{AL, dell,RS,russell} for the case $m\neq 0$.}
 As for the system \eqref{101} with Cattaneo's type heat conduction, Fern{\'a}ndez Sare, Liu, and Racke \cite{liure} investigated the exponential
stability region of the system \eqref{101}
when parameters $\alpha, \;\beta,\;\gamma $ satisfy certain {assumptions}.
Recently, \cite{dhkz,han} investigated the exponential stability and optimal polynomial {stability} of \eqref{101} with {and} without inertial term when $(\alpha, \;\beta,\;\gamma) \in Q^c :=[0,1]\times[0,1]\times (0,1]\setminus Q$.
More precisely, { the region $Q^c$ was divided into several subregions. Within each subregion, comprehensive spectrum analysis and resolvent estimation were conducted under varying conditions, such as when the inertial parameter $m$ is greater than zero or equal to zero, and when the wave speed is the same ($\sigma\tau=1$) or not ($\sigma\tau\neq 1$).}

 For the case {where} $(\alpha, \;\beta,\;\gamma) \in Q $, it is easy to know that zero is a spectrum point (see Section \ref{sec2}). We shall use the results in \cite{BCT} to obtain the polynomial stability of the system by proving the estimation of the {resolvent of the corresponding semigroup generator} both at infinity and near zero.  Our analysis
  includes the polynomial decay rate for the corresponding semigroup {when the system is with and without inertial term ($m>0$ and $=0$), respectively.}

The paper is organized as follows. In  Section \ref{sec2},  the preliminaries and the main results
of this paper are given.   We prove our main results for cases including and excluding the inertial term in Section \ref{mneq0} and \ref{m=0},
respectively. Section \ref{example} is dedicated to presenting applications to our results.

 \section{Preliminaries and main resuls}\label{sec2}
  Define a Hilbert space
\begin{equation}
\label{h}\mathcal{H}:=\mathcal{D}(A^{\frac{1}{2}})\times \mathcal{D}(A^{\frac{\gamma}{2}})\times H\times H
\end{equation}
with the inner product
$$ \langle U_1,U_2\rangle_{\mathcal{H}}=\sigma(A^{\frac{1}{2}}u_1,A^{\frac{1}{2}}u_2)+m(A^{\frac{\gamma}{2}}v_1, A^{\frac{\gamma}{2}}v_2)+(v_1,v_2)+(\theta_1, \theta_2)+\tau (q_1,q_2),$$
where
$U_{i}=(u_{i},v_{i},\theta_{i},q_{i})^{\top}\in\mathcal{H},\;i=1,2$.  Define an operator $\mathcal{A}:{\mathcal {D}}(\mathcal{ A}) \subseteq \mathcal{H}\to \mathcal{H}$ as \begin{equation}
\label{a}
\mathcal{A}\left[\begin{array}{c}
u\\
v\\
\theta\\
q
\end{array} \right]=\left(\begin{array}{c}
v\\
-(I+mA^{\gamma})^{-1}(\sigma A  u-A^{\alpha}\theta)\\
-A^\alpha v+A^{\frac{\beta}{2}}q\\ \displaystyle
{1\over \tau}(-q-A^{\frac{\beta}{2}}\theta)
\end{array}\right),
\end{equation}
    with domain
$$\mathcal{D}(\mathcal{A})=\left\{ (u, v, \theta, q)^{\top}\in \mathcal{H}\,\left|\,
\begin{array}{l}
v\in \mathcal{D}(A^{1\over2}),\;   \sigma A  u-A^\alpha\theta \in \mathcal{D}(A^{ -\frac{\gamma}{2}}),\;\\
	-A^\alpha v+A^{\frac{\beta}{2}}q\in H,\;
q+A^{\frac{\beta}{2}}\theta\in H
\end{array}
\right.\right\}.$$

Then {system} \eqref{101} can be written as an abstract first-order evolution equation:
\begin{equation}\label{105}
	\left\{ \begin{array}{l}
		\displaystyle {d\over dt}  U(t)=\mathcal{A} U(t) ,\quad t>0,\\ \noalign{\medskip}  \displaystyle
		U(0)=U_0{\in\mathcal{H}}. \end{array} \right.
\end{equation}

\begin{figure}
    \centering
        \includegraphics[width=7cm]{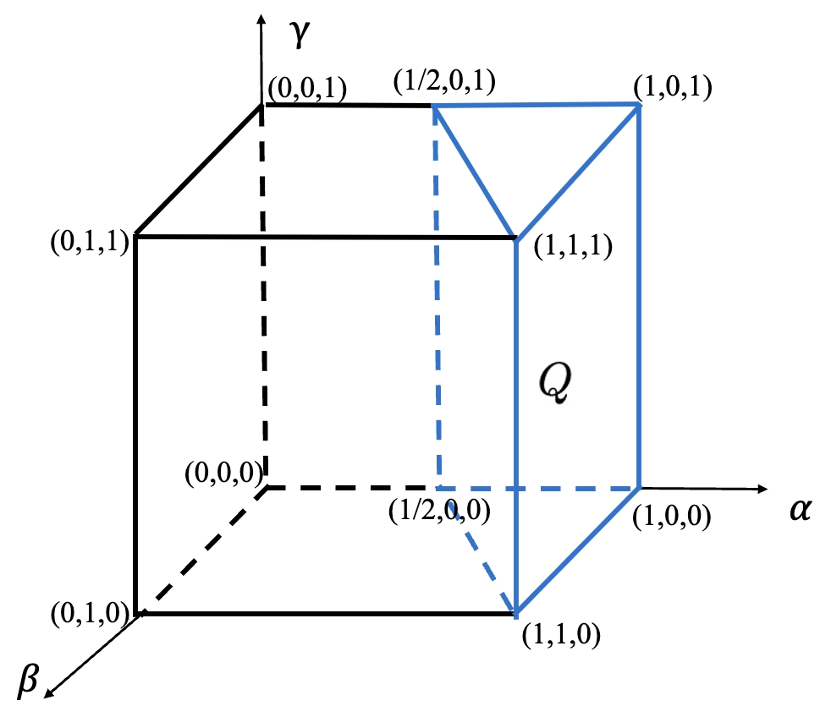}
    \caption{The region of $Q$ (blue triangular prism).}
    \label{12024}
\end{figure}

The closedness of $\mathcal{A}$ is not obvious, we thereby have the following Lemma.
\begin{lemma}\label{4.1}
    	Let $m\ge 0$ and  $(\alpha,\beta,\gamma)\in Q$, $\mathcal{H}$ and $\mathcal{A}$ be defined by \eqref{h} and \eqref{a}, respectively. Then ${\mathcal{A}}$ is closed.
\end{lemma}
\begin{proof}
    Assume that there exists $U_n=(u_n,v_n,\theta_n,q_n)^{\top}\in \mathcal{D}(\mathcal{A} )$ such that $U_n\to U_0:=(u_0,v_0,\theta_0,q_0)^{\top}$ in $\mathcal{H}$, $\mathcal{A}U_n\to W:=(w_1,w_2,w_3,w_4)^{\top}$ in $\mathcal{H}$, i.e., $$\left(\begin{array}{c}
v_n\\
-(I+mA^{\gamma})^{-1}(\sigma A  u_n-A^{\alpha}\theta_n)\\
-A^\alpha v_n+A^{\frac{\beta}{2}}q_n\\ \displaystyle
{1\over \tau}(-q_n-A^{\frac{\beta}{2}}\theta_n)
\end{array}\right)\to\left(\begin{array}{c}
w_1\\
w_2\\
w_3\\
w_4
\end{array}\right), \quad\text{in } \mathcal{D}(A^{\frac{1}{2}})\times \mathcal{D}(A^{\frac{\gamma}{2}})\times H\times H.$$
It suffices to show $U_0\in \mathcal{D}(\mathcal {A})$ and $\mathcal{A}U_0=W$.

By the assumption we have $v_n\to w_1$ in $\mathcal{D}(A^{1/2})$, $v_n\to v_0$ in $\mathcal{D}(A^{\gamma/2})$, then $v_0=w_1$.

Since $U_n\in \mathcal{D}(\mathcal{A})$, which implies $-q_n-A^{\beta/2}\theta_n\in H$, then {$B_n:=-A^{-\beta/2} q_n-\theta_n\in \mathcal{D}(A^{\beta/2})$}. Recalling that $A$ is a self-adjoint, positive definite operator on $H$,  and $ q_n\to q_0$ in $H$; therefore, $A^{-\beta/2} q_n\to A^{-\beta/2}q_0$ in $H$, which together with $\theta_n\to \theta_0$ in $H$ implies
 $$B_n\to -A^{-\beta/2} q_0-\theta_0 \quad \text{in } H.$$
 Moreover, by the assumption we have $A^{\beta/2}B_n\to \tau w_4$ in $H$. Note that $A^{\beta/2}$ is closed, we get $-q_0-A^{\frac{\beta}{2}}\theta_0=\tau w_4$.

By a similar argument as above, we can prove that $-A^\alpha v_0+A^{\frac{\beta}{2}}q_0=w_3$ and $-(I+mA^{\gamma})^{-1}(\sigma  A u_0-A^{\alpha}\theta_0)=w_2$. In conclusion,
$$\left(\begin{array}{c}
v_0\\
-(I+mA^{\gamma})^{-1}(\sigma A  u_0-A^{\alpha}\theta_0)\\
-A^\alpha v_0+A^{\frac{\beta}{2}}q_0\\ \displaystyle
{1\over \tau}(-q_0-A^{\frac{\beta}{2}}\theta_0)
\end{array}\right)=\left(\begin{array}{c}
w_1\\
w_2\\
w_3\\
w_4
\end{array}\right)\in \mathcal{H},$$
which further implies that  $U_0 \in \mathcal{D}(\mathcal{A})$. The proof is complete.
\end{proof}

 The well-posedness of the system (\ref{105}) is stated as follows.
\begin{theorem}\label{t-1-1}
	Let $m>0$ and  $(\alpha,\beta,\gamma)\in Q$, $\mathcal{H}$ and $\mathcal{A}$ be defined by \eqref{h} and \eqref{a}, respectively.  Then $\mathcal{A}$ generates a $C_{0}$-semigroup $(T(t))_{t\ge 0}$ of contractions on  $\mathcal{H}$.
\end{theorem}

\begin{proof}
We have shown in Lemma \ref{4.1} that $\mathcal{A}$ is a closed operator.
Note that $\mathcal{D}(\mathcal{A})\supseteq \mathcal{D}_0:=\mathcal{D} (A^{1-\frac{\gamma}{2}})\times\mathcal{D} (A^{\alpha}) \times\mathcal{D} (A^{\alpha-\frac{\gamma}{2}})\times \mathcal{D} (A^{\frac{\beta}{2}})$, and ${ {\overline{\mathcal{D}_0}}=\mathcal{H}} $, thus  $\mathcal{A}$ is densely defined on $\mathcal{H}$.
Furthermore, for any vector $U = (u, v, \theta, q)^\top\in \mathcal{D}(\mathcal{A})$, we have
\begin{align}\label{824}
\Re\langle\mathcal{A}U, U\rangle_{\mathcal{H}} = -\|q\|^2,
\end{align}
indicating that $\mathcal{A}$ is dissipative. By direct calculation, one gets $\mathcal{A}^*:\mathcal{D}(\mathcal{A}^*)\to \mathcal{H} $ is dissipative as well. By \cite[Corollary I.4.4]{pazy}, $\mathcal{A}$ is the generator of a $C_{0}$-semigroup $(T(t))_{t\ge 0}$ of contraction on $\mathcal{H}$.
\end{proof}

We now show that 0 {is} the unique spectral point {of $\mathcal{A}$} on the imaginary axis when $\alpha >\frac{\beta+1}{2}$.
 \begin{theorem}\label{1202}
   {  Let $m>0$ and  $(\alpha,\beta,\gamma)\in Q$, $\mathcal{H}$ and $\mathcal{A}$ be defined by \eqref{h} and \eqref{a}, respectively.}  Then $\sigma(\mathcal{A})\cap i\mathbb{R}=\{0\}$.
 \end{theorem}

\begin{proof}   We first claim that ${\mathcal{A}}$ is not surjective, then $0\in\sigma(\mathcal{A})$. Otherwise, for any $G:=(g_1,g_2,g_3,g_4)^{\top}\in \mathcal{H}$, there exists
 $U=(u,v,\theta,q)^{\top}\in \mathcal{D}(\mathcal{A})$ such that $\mathcal{A}U=G$.
 Solving the equation gives
\begin{align*}
  v =g_1, \quad
 \theta =-A^{-\frac{\beta}{2}}(\tau g_4+A^{-\frac{\beta}{2}}(g_3+A^{\alpha}g_1)), \quad   q =A^{-\frac{\beta}{2}}(g_3+A^{\alpha}g_1),
\end{align*}
and $$
u =\sigma^{-1}A^{-1}(A^{\alpha}\theta-(I+mA^{\gamma})g_2).
$$
However, from the equation of $q$, one has $A^{\alpha-\frac{\beta}{2}}g_1\in \mathcal{H}$, which is in contradiction with the arbitrariness of  $g_1\in \mathcal{D}(A^{\frac{1}{2}})$. Therefore,  ${\mathcal{A}}$ is not surjective.

 We proceed to prove $ i\mathbb{R}\backslash\{0\}\subseteq \rho(\mathcal{A})$. Suppose that there exists a $\lambda\in \mathbb{R}\backslash\{0\}$ such that $i\lambda\in \sigma(\mathcal{A})$.
	 Then there  exists  a sequence
	$U_n=(u_{n}, v_{n}, \theta_{n}, q_{n})^{\top}\subseteq\mathcal{D}(\mathcal{A})$ with
\begin{align}\label{+2111}
    \|U_n\|_\mathcal{H}=1,\;\; \forall\; n\in \mathbb{N},
\end{align}
 such that
	\begin{equation*}
		\Vert (i\lambda I-\mathcal{A}) U_n\Vert_\mathcal{H} =o(1),\quad n\to\infty,
	\end{equation*}
	i.e.,
	\begin{eqnarray}
		&&i\lambda A^{\frac{1}{2}}u_n- A^{\frac{1}{2}}v_n=o(1) \quad  \mbox{in } \;  H,\label{+g1}\\
		&&A^{-\frac{\gamma}{2}}\left(i\lambda v_n+i\lambda mA^{\gamma}v_n+\sigma A u_n-A^{\alpha}\theta_n\right)=o(1)\quad  \mbox{in } \;   H,\label{+g2}\\
		&&i\lambda\theta_n+A^\alpha v_n-A^{\frac{\beta}{2}}q_n=o(1)\quad  \mbox{in } \;   H,\label{+g3}\\
		&&i\lambda\tau q_n+q_n+A^{\frac{\beta}{2}}\theta_n=o(1)\quad  \mbox{in } \;   H.\label{+g4}
	\end{eqnarray}

{By a direct calculation we get }
	\begin{equation} \label{+k1-}
		\|q_n\|^2=-\Re\langle\mathcal{A}{U_n},{U_n}\rangle_{\mathcal{H}}=\Re\langle(i\lambda I-\mathcal{A}) U_n, U_n\rangle_\mathcal{H}=o(1).
	\end{equation}
Combining (\ref{+g4}) and \eqref{+k1-}, one has
\begin{equation}\label{+209+}
		\|A^{\frac{\beta}{2}}\theta_n\|=o(1), \;\; \|\theta_n\|=o(1).
	\end{equation}
 Note that $\alpha>\frac{\beta+1}{2}$, we obtain the following identity from \eqref{+g3}:
 \begin{align}\label{8241}
     i\lambda A^{\frac{\gamma}{2}-\alpha}\theta_n+A^{\frac{\gamma}{2}} v_n-A^{\frac{\beta+\gamma}{2}-\alpha}q_n=o(1).
 \end{align}
{We deduce from \eqref{+k1-}-(\ref{8241}) that}
	\begin{equation}\label{+216}
		\|A^{\frac{\gamma}{2}}v_n\|=o(1), \;\;\|v_n\|=o(1).
	\end{equation}
Taking the inner product of \eqref{+g3} with $\theta_n$ yields
\begin{align*}
    i\lambda\|\theta_n\|^2+(A^\alpha v_n,\theta_n)-(A^{\frac{\beta}{2}}q_n,\theta_n)=o(1){.}
\end{align*}
 {It} is easy to see from \eqref{+k1-} and (\ref{+209+}) and the above that
\begin{align}\label{1110}
    (A^\alpha\theta_n, v_n)=o(1).
\end{align}
Taking the inner product of \eqref{+g2} with $A^{\frac{\gamma}{2}}v_n$ on $H$,  combining with (\ref{+g1}), we get
	\begin{equation*}
		i\lambda\|v_n\|^2+i\lambda m\|A^{\frac{\gamma}{2}}v_n\|^2-i\lambda \sigma\|A^{\frac{1}{2}}u_n\|^2-(A^\alpha\theta_n,v_n)=o(1),
	\end{equation*}
 {which together with  \eqref{+216} and \eqref{1110}, yields}
\begin{equation}\label{+214+}
		\|A^{\frac{1}{2}}u_n\| =o(1).
	\end{equation}
Therefore, we arrive at the contradiction that {$\|U_n\|_{\mathcal{H}}=o(1)$} according to \eqref{+k1-}, \eqref{+209+}, \eqref{+216}, and \eqref{+214+}.
The proof is complete.
\end{proof}

By the same argument as above, we can prove the well-posedness of system \eqref{101} without inertial term, i.e., when $m=0.$
\begin{theorem}\label{t-1-1}
	Let $\mathcal{H}$ and $\mathcal{A}$ be defined by \eqref{h} and \eqref{a} with $m=0$ and  $(\alpha,\beta)\in Q^\ast$, where the region $Q^*$ is defined by
 (see Figure \ref{12025})
$$Q^*:=\Big\{(\alpha,\beta)\in [0,1]\times [0,1]\Big|\alpha>\frac{\beta+1}{2}\Big\}.$$
 Then $\mathcal{A}$ generates a $C_{0}$-semigroup $(T(t))_{t\ge 0}$ of contractions on  $\mathcal{H}$. Furthermore, $\sigma(\mathcal{A})\cap i\mathbb{R}=\{0\}$.
\end{theorem}

In this paper, we explore the long-time behavior of the solution \( T(\cdot)U_0 \) of system \eqref{105} for initial values \( U_0 \) belonging to \( \mathcal{D}({\mathcal{A}}) \cap \mathcal{R}(\mathcal{A}) \). Note that since $\mathcal{A}$ generates a $C_{0}$-semigroup $(T(t))_{t\ge 0}$ of contractions on  $\mathcal{H}$, $-\mathcal{A}$ is a sectorial operator and thereby $\mathcal{R}({\mathcal{A}}^{s}(I-{\mathcal{A}})^{-(s+t)})=\mathcal{R}({\mathcal{A}}^{s})\cap\mathcal{D}({\mathcal{A}}^{t})$ for $s,t\ge0$ by \cite[Proposition 3.10]{BCT}. The main results are stated as Theorems \ref{thm1} and \ref{thm2}.
\begin{theorem}\label{thm1}
  {  Let $m>0$ and  $(\alpha,\beta,\gamma)\in Q$, $\mathcal{H}$ and $\mathcal{A}$ be defined by \eqref{h} and \eqref{a}, respectively.}   Suppose that $(T(t))_{t\ge 0}$ is the $C_0$-semigroup of contractions generated by ${\mathcal{A}}$.
Then  for $t\ge 1,$
  \begin{align} \label{12021}
  \big\|T(t)
 {\mathcal{A}}(I-{\mathcal{A}})^{-\big(1+\frac{2(2\alpha-\beta-\gamma)}{2\alpha-\gamma}\big)}\big\|
  =O(t^{-1}), \end{align}  and
  \begin{align}\label{12022}
       \big\|T(t){\mathcal{A}}(I-{\mathcal{A}})^{-2}\big\|=O(t^{-\frac{1}{a}}),
        \quad  a=\max\Big\{1,\;\frac{2(2\alpha-\beta-\gamma)}{2\alpha-\gamma}\Big\}.
  \end{align}
 Moreover, the decay rate of \eqref{12022} is sharp if  $\alpha\ge\beta+\frac{\gamma}{2}$.
\end{theorem}

\begin{remark}
It is obvious that \eqref{12022} is a consequence of \eqref{12021} if $\alpha<\beta+\frac{\gamma}{2}$, the optimal decay rate of \eqref{12022} in this case is unclear.
\end{remark}

The following theorem gives the stability of the abstract coupled thermoelastic system \eqref{101} without inertial term.

\begin{figure}
    \centering
    \includegraphics[width=7cm]{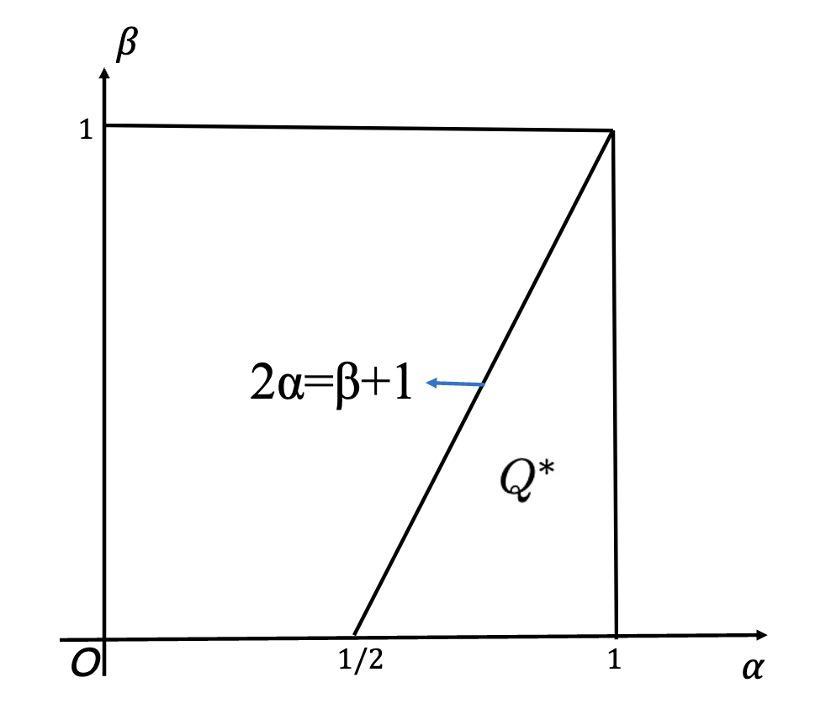}
    \caption{The region of $Q^*$.}
    \label{12025}
\end{figure}
\begin{theorem}\label{thm2}
   Let $m=0$ and  $( \alpha,\beta)\in Q^*$, $\mathcal{H}$ and $\mathcal{A}$ be defined by \eqref{h} and \eqref{a}, respectively.   Suppose that $(T(t))_{t\ge 0}$ is the $C_0$-semigroup of contractions generated by $\mathcal{A}$. Then 	 for $t\ge 1,$
  \begin{align} \label{12023}
      \big\|T(t){\mathcal{A}}(I-{\mathcal{A}})^{-2}\big\|=
      O(t^{-\frac{\alpha}{2\alpha-\beta}}).
  \end{align} Moreover, the decay rate of \eqref{12023} is sharp.
\end{theorem}

We shall prove Theorems \ref{thm1} and \ref{thm2} by estimating the norm of the corresponding resolvent operator along the imaginary axis, especially at zero and infinity on the imaginary axis. Our proof is based on the following result { on frequency characteristics for polynomial-stable semigroup}.

 \begin{lemma}(\cite[Theorem 8.4]{BCT})\label{0s}
    Let $(T(t))_{t\ge 0}$ be a bounded $C_0$-semigroup on a Hilbert space $\mathcal{H}$ with generator $\mathcal{A}$. Assume that $\sigma(\mathcal{A})\cap i\mathbb{R}=\{0\}$ and that there exist $l\ge 1$, $k>0$ such that
  \begin{align}\label{8232}
     \big \|(is I-\mathcal{A})^{-1}\big\|=\left\{ \begin{aligned}
    &O(|s|^{-l}),&s\to 0,\\
    &O(|s|^{k}),&|s|\to \infty.
    \end{aligned}\right.
  \end{align}
  Then
  \begin{align*} \big\|T(t)\mathcal{A}^{l}(I-\mathcal{A})^{-(l+k)}\big\|=O(t^{-1}),\;t\to \infty,\end{align*}  and
  \begin{align}\label{8231}
      \big \|T(t)\mathcal{A}(I-\mathcal{A})^{-2}\big\|=O(t^{-\frac{1}{a}}),\;t\to \infty,
  \end{align}
  where $a=\max\{l,k\}$.

  Conversely, if (\ref{8231}) holds for some $a>0$, then \eqref{8232} holds for $l=\max\{a,1\}$ and $k=a$.
\end{lemma}

 At the end of this section, we present the following interpolation lemma (see, e.g., \cite{dhkz} and \cite{liuzheng}), which will be utilized in subsequent sections.
\begin{lemma}\label{lemma-inter}
  Let $A \,:\, \mathcal{ D}(A) \subseteq H$ be self-adjoint and positive definite, $r,\;p\;,q\in \mathbb{R}$. Then
  \begin{align}\label{4.41}
      \|A^p x\| \le \|A^q x\|^{p-r\over q-r} \|A^r x\|^{q-p\over q-r},  \quad
\forall \;  r\le p \le q, \; x\in \mathcal{ D}(A^q).
  \end{align}
 \end{lemma}

\begin{remark}
   The study of stability of abstract thermoelastic system \eqref{101} has been extensively covered in \cite{dhkz,liure,han} and this paper. In \cite{liure}, the exponential stability region of system \eqref{101} was obtained. However, later in \cite{han}, Han, Kuang, and Zhang demonstrated that the exponentially stable region could be extended if $\sigma\tau = 1$. Furthermore, they investigated polynomial stability of system \eqref{101} when $(\alpha,\beta,\gamma)\in Q^c := [0,1] \times [0,1] \times (0,1] \setminus Q$ taking certain specific values. Following this track, \cite{dhkz} provided a detailed analysis of the polynomial stability of system \eqref{101} for every $(\alpha,\beta,\gamma)\in Q^c$. Finally, this paper considers the polynomial stability of system \eqref{101} when $(\alpha,\beta,\gamma)\in Q$.

In all the aforementioned papers, the proof techniques are similar, utilizing frequency domain methods combined with interpolation Lemma \ref{lemma-inter} and detailed spectral analysis. The key difference is that for the case $(\alpha,\beta,\gamma)\in Q^c$, the results were proved by contradiction based on \cite{bobtov1}, which requires $i\mathbb{R}\subseteq \rho(\mathcal{A})$. However, when $(\alpha,\beta,\gamma)\in Q$, since 0 is in the spectrum of the generator $\mathcal{A}$, we use Lemma \ref{0s} instead of \cite{bobtov1} to prove the result. Another difference is that the method used to prove the well-posedness of system \eqref{101} in this paper (Theorem \ref{t-1-1}) is relatively new.

\end{remark}

\section{Stability of system with inertial term (Proof of theorem \ref{thm1}) }
\label{mneq0}

In this section, we focus on analyzing the polynomial stability of the system \eqref{101}, considering \(m>0\), specifically aiming to prove Theorem \ref{thm1}. According to Lemma \ref{0s}, it is sufficient to show that for $k=\frac{2(2\alpha-\beta-\gamma)}{2\alpha-\gamma},\;l= 1$, the following holds:
\begin{align}    \label{7141}
\begin{array}{ll}
     \|s(isI-\mathcal{A})^{-1}\|&=O(1),\;\; s\to 0, \\
     \|\lambda^{-k}(i\lambda I-\mathcal{A})^{-1}\|&=O(1),\;\; \lambda\to \infty.
\end{array}
\end{align}

 If \eqref{7141} fails,  there exists  a sequence $(\eta_n, \lambda_{n},  U_{1,n}, U_{2,n}  )_{n\ge1}$, where $\eta_n:=s_n^{-1},\lambda_{n}\in {\mathbb R} ,  U_{j,n}:=(u_{j,n}, {v}_{j,n}, \theta_{j,n},\, q_{j,n})^{\top}\in \mathcal{D}({\cal A}),j=1,2$ satisfying
\begin{equation}\label{unitnorm}
\|U_{j,n}\|_{\cal H}  =\big\| (u_{j,n}, {v}_{j,n}, \theta_{j,n}, q_{j,n})^{\top}\big\|_{ \cal H }  =1,\; j=1,2,
\end{equation}
such that  as $ n\to\infty,\eta_{n},\lambda_{n}\to\infty$, and
 \begin{align}\label{8246}
   (i I-\eta_n{\mathcal{A}})U_{1,n}=o(1), \;\;\;\;\qquad
      \lambda_{n}^{k}\big( i\, \lambda_{n}I- {\mathcal{A}}\big)   U_{2,n} =o(1)
   \;   \mbox{ in  }  \; {\cal H}.
\end{align}
Equivalently, we have
\begin{eqnarray}
&&iA^{\frac{1}{2}}u_{1,n}-\eta_n A^{\frac{1}{2}}v_{1,n}=o(1)\quad  \;  \mbox{ in  }  \;  H,\label{0s3}\\   \noalign{\medskip}  \displaystyle
  &&A^{-\frac{\gamma}{2}}\left(iv_{1,n}+im A^{\gamma}v_{1,n}+\sigma \eta_nA u_{1,n}-\eta_nA^{\alpha}\theta_{1,n}\right)=o(1)\quad  \;  \mbox{ in  }  \;  H,\label{0s4}\\\noalign{\medskip}  \displaystyle
  &&i\theta_{1,n}-\eta_nA^{\frac{\beta}{2}}q_{1,n}+\eta_nA^\alpha v_{1,n}=o(1)\quad  \;  \mbox{ in  }  \;  H,\label{0s5}\\\noalign{\medskip}  \displaystyle
  &&i\tau q_{1,n}+\eta_nq_{1,n}+\eta_nA^{\frac{\beta}{2}}\theta_{1,n}=o(1)\quad  \;  \mbox{ in  }  \;  H.\label{0s6}
\end{eqnarray}
and
\begin{eqnarray}
&&\lambda_n^k(i\lambda_nA^{\frac{1}{2}}u_{2,n}- A^{\frac{1}{2}}v_{2,n})=o(1)\quad  \;  \mbox{ in  }  \;  H,\label{s1+}\\   \noalign{\medskip}  \displaystyle
  &&\lambda_n^kA^{-\frac{\gamma}{2}}\left(i\lambda_nv_{2,n}+i\lambda_n m A^{\gamma}v_{2,n}+\sigma A u_{2,n}-A^{\alpha}\theta_{2,n}\right)=o(1)\quad  \;  \mbox{ in  }  \;  H,\label{s2+}\\\noalign{\medskip}  \displaystyle
  &&\lambda_n^k(i\lambda_n\theta_{2,n}-A^{\frac{\beta}{2}}q_{2,n}+A^\alpha v_{2,n})=o(1)\quad  \;  \mbox{ in  }  \;  H,\label{s3+}\\\noalign{\medskip}  \displaystyle
  &&\lambda_n^k(i\lambda_n\tau  q_{2,n}+ q_{2,n}+A^{\frac{\beta}{2}}\theta_{2,n})=o(1)\quad  \;  \mbox{ in  }  \;  H.\label{s4+}
\end{eqnarray}
We shall prove $ \|U_{j,n}\|=o(1),\; j=1,2$, which contradicts to the assumption \eqref{unitnorm}. The proof is structured into two steps.

\noindent{\it Step 1.} We claim that $ \|U_{1,n}\|=o(1).$

By (\ref{824}) and the first identity of \eqref{8246}, it is easy to see
\begin{align}\label{0s7}
    \|q_{1,n}\|=\eta_n^{-\frac{1}{2}}o(1).
\end{align}
According to (\ref{0s6}) and \eqref{0s7}, we get
\begin{align}\label{0s8}
    \|A^{\frac{\beta}{2}}\theta_{1,n}\|=\eta_n^{-\frac{1}{2}}o(1).
\end{align}
Therefore,
\begin{align}\label{0s9}
    \|\theta_{1,n}\|=o(1).
\end{align}

Combining $\alpha>\frac{1}{2}$ and $\eqref{unitnorm}$, one has $\|A^{-\alpha+\gamma}v_{1,n}\|$ is bounded. Thus, taking the inner product of (\ref{0s5}) with $\eta^{-1}_nA^{-\alpha+\gamma}v_{1,n}$, we have
\begin{align}\label{0s10}
    i\eta^{-1}_{n}(\theta_{1,n},A^{-\alpha+\gamma}v_{1,n})-(A^{\frac{\beta}{2}}q_{1,n},A^{-\alpha+\gamma}v_{1,n})+\|A^{\frac{\gamma}{2}}v_{1,n}\|^2=o(1).
\end{align}
The first term of (\ref{0s10}) tends to 0 because of \eqref{0s9} and the boundedness of  $\|A^{-\alpha+\gamma}v_{1,n}\|$. The second term of (\ref{0s10}) tends to 0 because of \eqref{0s7} and   $-\alpha+\frac{\beta}{2}+\gamma<\frac{\gamma}{2}$. Therefore,
\begin{align}\label{0s11}
    \|A^{\frac{\gamma}{2}}v_{1,n}\|=o(1),\; \;  \mbox{and then   }  \;   \|v_{1,n}\|=o(1).
\end{align}

From (\ref{0s5}), we see $  i\eta_n^{-\frac{1}{2}}A^{-\frac{\beta}{2}}\theta_{1,n}-\eta_n^{\frac{1}{2}}q_{1,n}+\eta_n^{\frac{1}{2}}A^{\alpha-\frac{\beta}{2}}v_{1,n}=o(1).$ Thus, by \eqref{0s7} and \eqref{0s9}, we get
\begin{align}\label{0s12}
    \|\eta_n^{\frac{1}{2}}A^{\alpha-\frac{\beta}{2}}v_{1,n}\|=o(1).
\end{align}
Taking the inner product of \eqref{0s4} with $A^{\frac{\gamma}{2}}v_{1,n}$ on $H$ yields
	\begin{equation}\label{8253}
		i\|v_{1,n}\|^2+im\|A^{\frac{\gamma}{2}}v_{1,n}\|^2+(\sigma A^{\frac{1}{2}}u_{1,n},\eta_nA^{\frac{1}{2}}v_{1,n})-(\eta_n A^\alpha \theta_{1,n},v_{1,n})=o(1).
	\end{equation}
By \eqref{0s3} and \eqref{8253}, we get
 \begin{equation}\label{0s13}
		i\|v_{1,n}\|^2+im\|A^{\frac{\gamma}{2}}v_{1,n}\|^2-i\sigma \|A^{\frac{1}{2}}u_{1,n}\|^2-(\eta_n A^\alpha \theta_{1,n},v_{1,n})=o(1).
	\end{equation}
Recalling  \eqref{0s8} and \eqref{0s12}, one has
	\begin{equation*}
		(\eta_n A^\alpha\theta_{1,n},v_{1,n})\leq\|\eta_n^{\frac{1}{2}}A^{\frac{\beta}{2}}\theta_{1,n}\|\|\eta_n^{\frac{1}{2}}A^{\alpha-\frac{\beta}{2}}v_{1,n}\|=o(1).
	\end{equation*}

Combining this, (\ref{0s11}) and (\ref{0s13}), we obtain
	\begin{equation}\label{0s14}
		\|A^{\frac{1}{2}}u_{1,n}\| =o(1).
	\end{equation}
	In summary, by \eqref{0s7}, \eqref{0s9}, \eqref{0s11} and \eqref{0s14}, we conclude
	$$\|U_{1,n}\|_{\mathcal{H}}=\|(u_{1,n}, v_{1,n}, \theta_{1,n}, q_{1,n})\|_{\mathcal{H}}=o(1).$$

 \noindent{{\it Step 2.} We claim that $ \|U_{2,n}\|_{\mathcal{H}}=o(1)$.}

  By \eqref{824} and the second identity of \eqref{8246}, it is easy to see
\begin{align}\label{v1_0}
   \|q_{2,n}\|= \lambda_n^{-\frac{k}{2}} o(1).
\end{align}
According to (\ref{s4+}) and \eqref{v1_0}, we get
\begin{align}\label{v1_8}
   \| A^{\frac{\beta}{2}}\theta_{2,n}\|=\lambda_n^{1-\frac{k}{2}}o(1).
\end{align}

Since $\alpha>\frac{\beta+1}{2}$, \eqref{s3+} implies
$$ i\lambda_nA^{-\alpha+\frac{\gamma}{2}}\theta_{2,n}-A^{-\alpha+\frac{\gamma}{2}+\frac{\beta}{2}}q_{2,n}+A^{\frac{\gamma}{2}} v_{2,n} =\lambda_n^{-k}o(1).
$$
Recalling \eqref{unitnorm} and \eqref{v1_0}, we obtain from the above that
 \begin{equation}\label{v1_1}
 	\|\lambda_nA^{-\alpha+\frac{\gamma}{2}}\theta_{2,n}\|=O(1).
\end{equation}
{By interpolation we deduce from \eqref{v1_8} and \eqref{v1_1} that}
\begin{equation}\label{v1_2}
\left\| \theta_{2,n}\right\| \leq \|A^{\frac{\beta}{2}} \theta_{2,n}\|^{\frac{2\alpha-\gamma}{2\alpha+\beta-\gamma}} \| A^{-\alpha+\frac{\gamma}{2}} \theta_{2,n} \|^{\frac{\beta}{2\alpha+\beta-\gamma}}=o(1).
\end{equation}

Taking the inner product of \eqref{s2+} with $\lambda_n^{-k}A^{-\alpha+\frac{\gamma}{2}}\theta_{2,n}$, \eqref{s3+} with $\lambda_n^{-k}A^{-\alpha}(I+mA^{\gamma})v_{2,n}$ and then adding them, we get
\begin{align}\label{v1_5}
\begin{array}{lll}
    & {\Re}(\sigma\theta_{2,n},\, A^{1-\alpha}u_{2,n})-{\Re}{( q_{2,n},\, A^{\frac{\beta}{2}-\alpha}(I+mA^{\gamma})v_{2,n})}\\
&\quad -\|\theta_{2,n}\|^2+\|v_{2,n}\|^2+m\|A^{\frac{\gamma}{2}}v_{2,n}\|^2=\lambda_n^{-k}o(1).
\end{array}
\end{align}
The first two terms in \eqref{v1_5} tend to 0 because of \eqref{v1_0}, \eqref{v1_2} and the boundedness of $\|A^{\frac{1}{2}}u_{2,n}\|$ and $\|A^{\gamma\over 2} v_{2,n}\|$.
These, along with (\ref{v1_2}) and \eqref{v1_5}, imply
\begin{equation}\label{v1_10}
	\|v_{2,n}\|,~~\|A^{\frac{\gamma}{2}}v_{2,n}\| =o(1).
\end{equation}

One can deduce from \eqref{s2+} that
\begin{align*}
iA^{-\frac{1}{2}}v_{2,n}+ imA^{\gamma-\frac{1}{2}}v_{2,n}+\sigma\lambda_n^{-1}A^{\frac{1}{2}}u_{2,n}-\lambda_n^{-1}A^{\alpha-\frac{1}{2}}\theta_{2,n}=\lambda_n^{-k-1}o(1),
\end{align*}
{which along with}  \eqref{v1_10} and the boundedness of $\|A^{\frac{1}{2}}u_{2,n}\|$, {yields}
\begin{align}\label{8261}
    \|\lambda_n^{-1}A^{\alpha-\frac{1}{2}}\theta_{2,n}\|=o(1).
\end{align}
Now, taking  the inner product of \eqref{s2+}  with ${\lambda_n^{-k}}A^{\frac{\gamma}{2}}u_{2,n}$ on $H$, along with \eqref{s1+}, one has
 \begin{equation}\label{316-}
	-\|v_{2,n}\|^2-m\|A^{\frac{\gamma}{2}}v_{2,n}\|^2 +\sigma\|A^{\frac{1}{2}} u_{2,n}\|^2 -(A^{\alpha}\theta_{2,n}, u_{2,n})
	 = o(1).
\end{equation}
Note that by \eqref{s1+} and (\ref{8261}), we have
\begin{equation}\label{8262}
\begin{array}{ll}
     (A^{\alpha}\theta_{2,n},\, u_{2,n})&= i(\lambda_n^{-1}A^{\alpha}\theta_{2,n},\,v_{2,n})+o(1)\\
&= i(A^{\frac{\beta}{2}}\theta_{2,n},\,\lambda_n^{-1}A^{\alpha-\frac{\beta}{2}}v_{2,n})+o(1).
\end{array}
\end{equation}
Furthermore, by $\eqref{s3+}$, we see
\begin{equation}\label{8251}
\lambda_n^{-1}A^{\alpha-\frac{\beta}{2}} v_{2,n}= -iA^{-\frac{\beta}{2}}\theta_{2,n}+ \lambda_n^{-1}q_{2,n}+\lambda_n^{-k-1}o(1){.}
\end{equation}
 Substituting \eqref{8251} into \eqref{8262} yields
\begin{align*}
    &(A^{\alpha}\theta_{2,n},\, u_{2,n})
    =-\|\theta_{2,n}\|^2+ i(\lambda_n^{-1}A^{\frac{\beta}{2}}\theta_{2,n},\,q_{2,n})=o(1),
\end{align*}
where we use \eqref{v1_0}, \eqref{v1_8} and \eqref{v1_2}.
 Combining this, \eqref{v1_10} and (\ref{316-}), we obtain
\begin{equation}\label{m018}
	\|A^{\frac{1}{2}}u_{2,n}\|={o(1)}.
\end{equation}
Recalling \eqref{v1_0}, \eqref{v1_2}, \eqref{v1_10} and \eqref{m018}, we get $\|U_{2,n}\|_{\mathcal{H}}=o(1)$, which {contradicts} (\ref{unitnorm}). Therefore, {by the above two steps, we have proved that the assumption \eqref{7141} holds with $k=\frac{2(2\alpha-\beta-\gamma)}{2\alpha-\gamma},\;l= 1$.} As a result, {thanks to Lemma \ref{0s},} the semigroup $T(t) $ satisfies \eqref{12021}-\eqref{12022} when $t\to \infty$.

At the end of this section, we show that the decay order is sharp if $\alpha\ge\beta+\frac{\gamma}{2}$, i.e., $k\ge1$ by analyzing the eigenvalues.
Noticing that $A$ is a self-adjoint, positive definite operator with compact resolvent. Thus, there exists a sequence of eigenvalues $\{\m_n\}_{n\ge1}$  of $A$ such that
\begin{equation*}
0<\m_1\le\m_2\le\cdots\le\m_n\le\cdots,~~\qquad\lim_{n\to\infty}\m_n=\infty.
\end{equation*}
The eigenvalues $\lambda$ of operator ${\cal A}$ satisfy the following quartic equation:
\begin{equation}
\label{f=0}
\begin{array}{lll}
f(\l,\m_n) &\equiv (m\tau \cdot \mu_n^\gamma + \tau ) \lambda^4 + (m \cdot  \mu_n^\gamma +1) \lambda^3 \\
& \quad + (\tau \cdot \mu_n^{2\alpha} + m \cdot \mu_n^{\beta+\gamma} + \sigma \tau \cdot \mu_n +  \mu_n^\beta) \lambda^2\\
 & \quad  + (\mu_n^{2\alpha}+\sigma \cdot \mu_n) \lambda + \sigma \cdot \mu_n^{1+\beta}\\
&=0.
\end{array}
\end{equation}
Without loss of generality, choose $\sigma=2$ and $\tau=1$. Taking $m=1$ we derive the characteristic equation of the system from (\ref{f=0}):
\begin{equation}
\label{eq:characteristic-with-inertial}
(\mu_n^\gamma + 1 ) \lambda_n^4 + ( \mu_n^\gamma +1) \lambda_n^3 + (  \mu_n^{2\alpha} + \mu_n^{\beta+\gamma} + 2 \mu_n +  \mu_n^\beta) \lambda_n^2 + (\mu_n^{2\alpha}+2 \mu_n) \lambda_n + 2 \mu_n^{1+\beta}=0,
\end{equation}
where $\lambda_n$ is the eigenvalue of $\mathcal{A}$ associated with $\mu_n \rightarrow \infty$ as $n \rightarrow \infty$. We aim to solve \eqref{eq:characteristic-with-inertial}  in the setting $\mu_n \rightarrow \infty$. Using a similar method as shown in \cite{dhkz,kuang, han, haoliu2}, we get that  the solutions to \eqref{eq:characteristic-with-inertial} when $(\alpha,\beta,\gamma)\in Q$ are as follows:
\begin{align*}
\lambda_{1,n}&=-\frac{1}{2}\m_n^{-2\a+\beta+\g}(1+o(1)) + i\m_n^{\a-\frac{\g}{2}}(1+o(1)), \\
\lambda_{2,n}&=-\frac{1}{2}\m_n^{-2\a+\beta+\g}(1+o(1)) - i\m_n^{\a-\frac{\g}{2}}(1+o(1)), \\
\lambda_{3,n}&=			-2\m_n^{-2\a+\beta+1}(1+o(1)),\\
\lambda_{4,n}&=			-1(1+o(1)).
\end{align*}
It is easy to see from $\lambda_{1,n}$ and $\lambda_{2,n}$ that
\begin{equation}
\label{o1}
|\Re \lambda_{i,n} |= \frac{1}{2}| \Im   \lambda_{i,n} |^{-{ k}}, \quad
\;  \mbox{ when  }  \; (\alpha, \; \beta,\; \gamma) \in Q,\; i=1,2.
\end{equation}

The following proposition gives the sharpness of the decay rate.
\begin{proposition}\label{12.4}
   Let { the} conditions in Theorem \ref{thm1} hold.  Suppose $\alpha\ge\beta+\frac{\gamma}{2}$. Then the decay rate in \eqref{12022} is sharp.
\end{proposition}
\begin{proof}
By the proof of Theorem \ref{thm1}, estimation (\ref{8232}) in Lemma \ref{0s} holds with $k=\frac{2(2\alpha-\beta-\gamma)}{2\alpha-\gamma},\;l=1$.
    If $\alpha\ge\beta+\frac{\gamma}{2}$, then $k\geq1$. Consequently, we have
    \begin{align*}
       \|T(t){\mathcal{A}}(I-{\mathcal{A}})^{-2}\|=O(t^{-\frac{1}{k}}),\;t\to \infty.
  \end{align*}
 We only prove the case when $k>1$, the case when $k=1$ is similar. If the decay rate can be improved, in other words, there exists $\varepsilon>0$ small enough such that $k-\varepsilon>1$ and $ \|T(t){\mathcal{A}}(I-{\mathcal{A}})^{-2}\|=O(t^{-\frac{1}{k-\varepsilon}})$, then by Lemma \ref{0s}, one has
  \begin{align*}
      \|(i\lambda I-{\mathcal{A}})^{-1}\|=\left\{ \begin{aligned}
    &O(|\lambda|^{-k+\varepsilon}),&\lambda\to 0,\\
    &O(|\lambda|^{k-\varepsilon}),&|\lambda|\to \infty.
    \end{aligned}\right.
  \end{align*}
In particular, there exists a constant $C>1$ such that
\begin{equation}
\label{c1}\|(i\lambda I-{\mathcal{A}})^{-1}\|\leq C|\lambda|^{k-\varepsilon},\;\; |\lambda|\to \infty.
\end{equation}
Let $S_{\lambda}:=\big\{r+i\lambda\,\big|\,|r|\leq \frac{1}{2C|\lambda|^{k-\varepsilon}},\; r,\;\lambda\in {\mathbb R\backslash\{0\}}\big\}$, then for any $s\in S_{\lambda}$,
  \begin{align*}
      \|(sI-{\mathcal{A}})^{-1}\|&=\|(i\lambda I-{\mathcal{A}})^{-1}(I+r(i\lambda I-{\mathcal{A}})^{-1})^{-1}\|\\
     & \leq \|(i\lambda I-{\mathcal{A}})^{-1}\|\frac{1}{1-\|r(i\lambda I-{\mathcal{A}})^{-1}\|}\\
     & \leq 2\|(i\lambda I-\mathcal{A})^{-1}\|\\
    &  \leq 2C|\lambda|^{k-\varepsilon},\;\;\;|\lambda|\to \infty,
  \end{align*}
  which implies that $S_{\lambda}\subseteq \rho({\mathcal{A}})$ for $|\lambda|$ big enough.

  On the other hand, recalling that there exists a sequence $(\lambda_n)_{n\ge 1}\subseteq\sigma({\mathcal{A}}),\;|\lambda_n|\to \infty$ such that \eqref{o1} holds. For the constant $C$ in \eqref{c1} and an arbitrary positive constant $\varepsilon$, we choose   $\lambda_n$   such that $|\Im \lambda_n|^{-\varepsilon} \leq \frac{1}{C}$, then
  $$|\Re \lambda_{n}| = \frac{1}{2}| \Im   \lambda_{n} |^{-{ k}}
  \leq \frac{1}{2C|\Im\lambda_n|^{k-\varepsilon}}.$$
  Thus, $\lambda_n\in S_{\lambda_n}$ which contradicts $\lambda_n\in \sigma({\mathcal{A}})$.
\end{proof}

\section{Stability of system without inertial term (Proof of theorem \ref{thm2})}\label{m=0}
\setcounter{equation}{0}
In this section, we shall analyze the polynomial stability of system \eqref{101} without inertial term, i.e.,   prove Theorem \ref{thm2}. By Lemma  \ref{0s}, it is sufficient to show that \eqref{8232} holds with $k=\frac{2\alpha-\beta}{\alpha},\; l= 1$. Similar to the argument in Section \ref{mneq0},  we still prove this theorem by contradiction. Suppose (\ref{8232}) fails, then there at least exists  a sequence $\{ \eta_n,\lambda_{n}, U_{1,n},U_{2,n}  \}_{n=1}^\infty
\subseteq {\mathbb R}^2\times \mathcal{D}({\cal A})^2  $ such that \eqref{unitnorm}-\eqref{8246} hold with $m=0$. In other words, we have
\begin{eqnarray}
&&iA^{\frac{1}{2}}u_{1,n}- \eta_nA^{\frac{1}{2}}v_{1,n}=o(1)\quad  \;  \mbox{ in  }  \;  H,\label{0s3+}\\   \noalign{\medskip}  \displaystyle
  &&iv_{1,n}+\sigma\eta_n A u_{1,n}-\eta_nA^{\alpha}\theta_{1,n}=o(1)\quad  \;  \mbox{ in  }  \;  H,\label{0s4+}\\\noalign{\medskip}  \displaystyle
  &&i\theta_{1,n}-\eta_nA^{\frac{\beta}{2}}q_{1,n}+\eta_nA^\alpha v_{1,n}=o(1)\quad  \;  \mbox{ in  }  \;  H,\label{0s5+}\\\noalign{\medskip}  \displaystyle
  &&i\tau q_{1,n}+ \eta_nq_{1,n}+\eta_nA^{\frac{\beta}{2}}\theta_{1,n}=o(1)\quad  \;  \mbox{ in  }  \;  H.\label{0s6+}
\end{eqnarray}
and
\begin{eqnarray}
&&\lambda_n^k(i\lambda_nA^{\frac{1}{2}}u_{2,n}- A^{\frac{1}{2}}v_{2,n})=o(1)\quad  \;  \mbox{ in  }  \;  H,\label{s1+1}\\   \noalign{\medskip}  \displaystyle
  &&\lambda_n^k \left(i\lambda_nv_{2,n} +\sigma A u_{2,n}-A^{\alpha}\theta_{2,n}\right)=o(1)\quad  \;  \mbox{ in  }  \;  H,\label{s2+1}\\\noalign{\medskip}  \displaystyle
  &&\lambda_n^k(i\lambda_n\theta_{2,n}-A^{\frac{\beta}{2}}q_{2,n}+A^\alpha v_{2,n})=o(1)\quad  \;  \mbox{ in  }  \;  H,\label{s3+1}\\\noalign{\medskip}  \displaystyle
  &&\lambda_n^k(i\lambda_n\tau  q_{2,n}+ q_{2,n}+A^{\frac{\beta}{2}}\theta_{2,n})=o(1)\quad  \;  \mbox{ in  }  \;  H.\label{s4+1}
\end{eqnarray}

We are devoted to showing that $ \|U_{j,n}\|_{\mathcal{H}}=o(1),\; j=1,2$, which {contradicts} the assumption \eqref{unitnorm}.
We first prove $ \|U_{1,n}\|_{\mathcal{H}}=o(1).$ Recalling that $\mathcal{A}$ is dissipative, \eqref{8246} and (\ref{0s6+}), we see
\begin{align}\label{0s7+}
    \|q_{1,n}\|=\eta_n^{-\frac{1}{2}}o(1), \; \; \|A^{\frac{\beta}{2}}\theta_{1,n}\|=\eta_n^{-\frac{1}{2}}o(1).
\end{align}
Therefore,
\begin{align}\label{0s9+}
    \|\theta_{1,n}\|=o(1).
\end{align}

Taking the inner product of (\ref{0s5+}) with $\eta^{-1}_nA^{-\alpha}v_{1,n}$, we have
\begin{align}\label{0s10+}
    i\eta^{-1}_{n}(\theta_{1,n},A^{-\alpha}v_{1,n})-(A^{\frac{\beta}{2}}q_{1,n},A^{-\alpha}v_{1,n})+\|v_{1,n}\|^2=o(1).
\end{align}
By Cauchy-Schwarz inequality and \eqref{0s7+}-\eqref{0s10+}, we get
\begin{align}\label{0s11+}
 \|v_{1,n}\|=o(1){.}
\end{align}
 Repeating the proof of \eqref{0s12}, we can deduce from \eqref{0s5+} and \eqref{0s7+} that
\begin{align}\label{0s12+}
    \|A^{\alpha-\frac{\beta}{2}}v_{1,n}\|=\eta_n^{-\frac{1}{2}}o(1).
\end{align}
Taking the inner product of \eqref{0s4+} with $v_{1,n}$ on $H$, along with \eqref{0s3+}, one has
	\begin{equation}\label{0s13+}
		i\|v_{1,n}\|^2-i\sigma\|A^{\frac{1}{2}}u_{1,n}\|^2-(\eta_n A^\alpha \theta_{1,n},v_{1,n})=o(1).
	\end{equation}
By \eqref{0s7+} and \eqref{0s12+}, the last term of (\ref{0s13+}) goes to 0 as $n\to \infty$. This together with (\ref{0s11+}) implies
\begin{equation}\label{0s14+}
		\|A^{\frac{1}{2}}u_{1,n}\| =o(1).
	\end{equation}
	In summary, by \eqref{0s7+}, \eqref{0s9+}, \eqref{0s11+} and \eqref{0s14+}, we obtain
	$$\|U_{1,n}\|_{\mathcal{H}}=\|(u_{1,n}, v_{1,n}, \theta_{1,n}, q_{1,n})\|_{\mathcal{H}}=o(1).$$

We proceed to prove $\|U_{2,n}\|_{\mathcal{H}}=o(1)$. We obtain from (\ref{824}) and (\ref{s4+1}) that
 \begin{align}\label{7142}
   \|q_{2,n}\|= \lambda_n^{-\frac{k}{2}} o(1),\;\; \| A^{\frac{\beta}{2}}\theta_{2,n}\|=\lambda_n^{1-\frac{k}{2}}o(1)
\end{align}
 as in Section 3. Note that $-\alpha+\frac{\beta}{2}< 0$, then by \eqref{unitnorm} and \eqref{s3+1}, {we get}
 \begin{equation*}
 \|A^{-\alpha}\theta_{2,n}\|=\lambda_n^{-1}O(1).
 \end{equation*}
Combining this and \eqref{7142} yields
 \begin{equation}\label{717}
 	\left\| \theta_{2,n}\right\| \leq \left\| A^{-\alpha} \theta_{2,n}\right\| ^{\frac{\beta}{2\alpha+\beta}}\|A^{\frac{\beta}{2}} \theta_{2,n}\|^{\frac{2\alpha}{2\alpha+\beta}}=o(1).
 \end{equation}

  It follows from \eqref{s2+1} that
\begin{align*}
    i \lambda_nA^{-\alpha} v_{2,n} +\sigma A^{1-\alpha}u_{2,n} - \theta_{2,n}=\lambda_n^{-k}o(1).
\end{align*}
Since $\alpha>\frac{1}{2}$, combining the above, \eqref{unitnorm} and \eqref{717}, one has
\begin{align}\label{8263}
     \| \lambda_nA^{-\alpha} v_{2,n}\|=O(1).
\end{align}

Taking the inner product of \eqref{s3+1} with $\lambda_{n}^{-k}A^{-\alpha}v_{2,n}$ yields
\begin{equation}\label{m06}
	(i\lambda_n\theta_{2,n}, A^{-\alpha}v_{2,n})-(A^{\frac{\beta}{2}-\alpha}q_{2,n}, v_{2,n})+\|v_{2,n}\|^2 =o(1).
\end{equation}

We see the first term of (\ref{m06}) tends to 0 because of \eqref{717} and \eqref{8263}, and the second term tends to 0 because of  \eqref{7142}. Thus,
\begin{equation}\label{m07}
\|v_{2,n}\|= o(1).
\end{equation}

Moreover, by taking the inner product of \eqref{s2+1} with $\lambda_n^{-k-1}v_{2,n}$, together with \eqref{s1+1}, we get
 \begin{equation*}
	i\|v_{2,n}\|^2-i\sigma\|A^{\frac{1}{2}} u_{2,n}\|^2 -(\theta_{2,n}, \lambda_n^{-1}A^{\alpha}v_{2,n})
	 = o(1).
\end{equation*}
By \eqref{s3+1}, we have $\lambda_{n}^{-1}A^{\alpha}v_{2,n}=-i\theta_n+\lambda_{n}^{-1}A^{\frac{\beta}{2}}q_{2,n}+\lambda_{n}^{-1-k}o(1)$. Substituting this into the above equation yields
\begin{equation}\label{8265}
	i\|v_{2,n}\|^2-i\sigma\|A^{\frac{1}{2}} u_{2,n}\|^2 -i\|\theta_{2,n}\|^2- (\lambda_n^{-1}A^{\frac{\beta}{2}}\theta_{2,n}, q_{2,n})
	 = o(1).
\end{equation}
 Therefore, we conclude from \eqref{7142}, \eqref{717},  \eqref{m07} and (\ref{8265}) that
\begin{equation}\label{8267}
	\|A^{\frac{1}{2}}u_{2,n}\|={o(1)}.
\end{equation}
Recalling \eqref{7142}, \eqref{717}, \eqref{m07} and \eqref{8267}, we get $\|U_{2,n}\|_{\mathcal{H}}=o(1)$, which {contradicts to} (\ref{unitnorm}).  Therefore,  the assumption \eqref{8232} holds with $k=\frac{2\alpha-\beta}{\alpha},\; l= 1$.

Finally,  we shall prove the decay order is sharp by a similar argument as  Section \ref{mneq0}. Note that $k=\frac{2\alpha-\beta}{\alpha}\ge 1$ {always holds}. Suppose   $\mu_n,\; \lambda_n,\;  n=1,2,\cdots$ are the eigenvalues of operators $A$ and ${\cal A}$, respectively.  Then repeat the process at the end part of Section \ref{mneq0}, we have the following quartic equation by taking $m=0,\sigma=2,\tau=1$ in \eqref{f=0},
\begin{equation}
\label{eq:characteristic-without-inertial}
\lambda_n^4 + \lambda_n^3 + ( \mu_n^{2\alpha} +  2 \mu_n +  \mu_n^\beta) \lambda_n^2 + (\mu_n^{2\alpha}+ 2 \mu_n) \lambda_n + 2 \mu_n^{1+\beta}=0.
\end{equation}
The solutions to \eqref{eq:characteristic-without-inertial} when $(\alpha,\beta)\in Q^*$ are the following:
\begin{align*}
\lambda_{1,n}&=-\frac{1}{2}\m_n^{\beta-2\a}(1+o(1))+ i\m_n^{\a}(1+o(1)), \\
\lambda_{2,n}&=-\frac{1}{2}\m_n^{\beta-2\a}(1+o(1))-i\m_n^{\a}(1+o(1)), \\
\lambda_{3,n}&=			-2\m_n^{-2\a+\beta+1}(1+o(1)),\\
\lambda_{4,n}&=			-1(1+o(1)).
\end{align*}

It is clear that
\begin{equation}
\label{o11}
|\Re \lambda_{i,n}| =\frac{1}{2} | \Im   \lambda_{i,n} |^{-{ k}}, \quad
\;  \mbox{ when  }  \; (\alpha, \beta) \in Q^*,\; i=1,2.
\end{equation}

Therefore, by the same argument as   Proposition \ref{12.4}, one can obtain that the decay rate in \eqref{12023} is sharp.
%
\section{Examples}\label{example}
\setcounter{equation}{0}

Assume that $\Omega$ is a bounded open subset of $\mathbb{R}^n$ with smooth  boundary  {$\Gamma$.} Let $A=\Delta^2$ be the bi-Laplace operator on $\Omega$ with domain $ \mathcal{D}(A) = \{u\in H^4(\Omega)\,|\,
u |_{ \Gamma }=  \Delta u |_{\Gamma } =0
\},\; H= L^2(\Omega),$ $  \alpha = 1 , \, \beta=0,\, \gamma={1\over2}$. Then the abstract system \eqref{101} can be written as follows:
\begin{equation}\label{602}
\left\{\begin{array}{ll}
u_{tt} - m \Delta u_{tt} + \sigma \Delta^2 u -\Delta^2\theta=0, &  x\in\Omega, t>0,\\
\theta_{t} - q+  \Delta^2 u_{t} = 0,&  x\in\Omega, t>0, \\
\tau q_t + q +\theta =0 , &   x\in\Omega, t>0,\\
u  =\Delta u = \theta =  q  =0 , &  x\in\Gamma, t>0,\\
u(0) =  u_0, \; u_t(0)=u_1,  \; \theta(0)=\theta_0, \;  q(0)=q_0, &  x\in\Omega.
\end{array}\right.
\end{equation}

By Theorems \ref{1202}, \ref{t-1-1}, \ref{thm1} and \ref{thm2}, we obtain that zero belongs to the spectrum of the generator of semigroup associated with \eqref{602} and   the semigroup  decays polynomially with
optimal order $t^{-{1\over2}}$ for  $m \ge0$.



\bibliographystyle{plain}
\bibliography{reference}
 \end{document}